\PassOptionsToPackage{reqno}{amsmath}
\documentclass[hidelinks,onefignum,onetabnum]{siamart251216}

\usepackage[T1]{fontenc}
\usepackage[utf8]{inputenc}
\usepackage{microtype}
\usepackage{amssymb,mathtools,mathrsfs}
\usepackage{bm}
\usepackage{booktabs}
\usepackage{array}
\usepackage{enumitem}
\usepackage{placeins}
\usepackage{subcaption}
\usepackage[numbers,sort&compress]{natbib}

\hypersetup{
  pdftitle={A second-order energy-stable low-regularity integrator for the Navier--Stokes equations},
  pdfauthor={Shu Ma},
  pdfkeywords={Navier--Stokes equations, exponential low-regularity integrator, energy stability, viscosity-uniform error estimate}
}

\graphicspath{{figures/}}
\allowdisplaybreaks
\newsiamthm{assumption}{Assumption}
\newsiamremark{remark}{Remark}
\crefname{assumption}{assumption}{assumptions}
\Crefname{assumption}{Assumption}{Assumptions}
\newcommand{\T}{\mathbb T}
\newcommand{\R}{\mathbb R}
\newcommand{\N}{\mathbb N}
\newcommand{\PP}{\mathbb P}
\newcommand{\ip}[2]{\left(#1,#2\right)}
\newcommand{\norm}[1]{\left\|#1\right\|}

\newcommand{\BM}{B_M}
\newcommand{\dd}{\,\mathrm d}
\DeclareMathOperator{\diver}{div}
\colorlet{BLUE}{blue}
\title{A second-order energy-stable low-regularity integrator for the Navier--Stokes equations
}
\author{
Shu Ma\thanks{Department of Mathematics, Hong Kong Baptist University,
Kowloon Tong, Hong Kong (\email{shu\_ma@hkbu.edu.hk}).}}
\headers{Energy-stable low-regularity integrator}{S. Ma}
\date{}

\usepackage{cases}

\begin{document}
\maketitle

\begin{abstract}
A second-order exponential low-regularity integrator is proposed for the incompressible Navier--Stokes equations. The method requires only linear solves and satisfies a discrete energy-decay property. Second-order temporal convergence is established under a uniform $H^4$ solution bound, with an error constant that remains bounded as the viscosity tends to zero. For the Fourier--Galerkin discretization, the optimal spatial convergence rate is established for fixed viscosity. The method can also be extended to conforming finite element methods while retaining the discrete energy-decay structure. Numerical results confirm the convergence orders and robustness with respect to the viscosity.
\end{abstract}



\section{Introduction}

We consider the periodic incompressible Navier--Stokes (NS) equations
\begin{equation}
\label{eq:NS}
\left \{
\begin{aligned} 
\partial_t u + u\cdot\nabla u - \mu \varDelta u
+\nabla p &= 0 && \text{in } \T^d\times (0,T],\\
\nabla\cdot u&=0&&\text{in } \T^d\times (0,T],\\
u(0)&=u_0&&\text{in } \T^d,
\end{aligned}
\right .
\end{equation}
where $d\in\{2,3\}$, $u$ and $p$ denote the velocity and pressure,
and $\mu>0$ is the viscosity. We write $u\cdot\nabla u=(u\cdot\nabla)u$. The well-posedness theory
in two and three dimensions is discussed in
\cite{FK64,KT01,LL11,Lions,Hou-NS-2021}. 
The NS equations model incompressible viscous flows. Since exact solutions are generally unavailable, the development and rigorous analysis of accurate and stable numerical methods are of fundamental importance.

A basic difficulty in the error analysis of time discretizations for
\eqref{eq:NS} is its dependence on the viscosity. Classical estimates often
use viscous dissipation to control the nonlinear terms. Young's inequality
then introduces negative powers of $\mu$, and the estimates deteriorate as
$\mu$ tends to zero. Viscosity-uniform estimates have been obtained for
second-order BDF-type schemes under suitable stabilization or step-size
conditions \cite{GAN,JGWZ}. For fixed viscosity, second-order BDF and linearly
extrapolated Crank--Nicolson methods have also been analyzed under reduced
initial regularity by using time-weighted estimates or refined time meshes
\cite{Emmrich,LMW,CWWZ}.

A separate issue is whether the time discretization retains the energy
structure of the equation. Second-order methods based on auxiliary variables,
original variables, splitting, and energy-conserving Runge--Kutta methods have
been studied in
\cite{LSL,JL,YTK,DHJL,WHWG,HS23,HS26,AWZ,Sanderse}. Depending on the
construction, the controlled quantity is either the original kinetic energy
or a modified discrete energy. Stability with respect to a modified energy
does not by itself imply decay of the original kinetic energy.

Second-order convergence under low regularity imposes an additional
constraint. A direct third-order Taylor expansion of \eqref{eq:NS} contains
$\mu^3\Delta^3u$, while an absorption argument for the convection defect may
introduce $\mu^{-1}$. Low-regularity integrators instead retain the linear
diffusion exactly and incorporate the nonlinear structure into the
consistency analysis \cite{RS}. For nonlinear Schr\"odinger
equations, related developments include optimal fully discrete EWI estimates
without a time--space coupling \cite{BaoWangEWI}, a symmetric second-order EWI
\cite{BaoWangSEWI}, and explicit symmetric low-regularity multistep schemes
\cite{FengMaierhoferWang}. For the Navier--Stokes equations, Li, Ma, and
Schratz \cite{LMS} constructed a first-order semi-implicit exponential method
that dissipates the original kinetic energy and has a
viscosity-independent error bound. Related exponential methods are studied
in \cite{JW,KBG,WWZ}. As far as we know, combining second-order accuracy, linear solvability, original energy decay, and viscosity-uniform error estimates under weak regularity remains difficult.

The aim is to construct a second-order linearly implicit method that preserves the
energy decay of the Navier--Stokes equations and remains accurate for small
viscosity. The method combines two exact heat half-steps with a
predictor--corrector approximation of convection. Each step requires two
linear frozen-advection solves and no nonlinear iteration, and is uniquely solvable
for every $\tau>0$. The convection corrector preserves energy, while
the heat half-steps give the exact dissipation identity in
Theorem~\ref{thm:energy}. Hence the original energy is nonincreasing
without a time-step restriction.

Assume that
$\norm{u}_{L^\infty(0,T;H^4)}\le R$.
For the Fourier cutoff $M$, Theorem~\ref{thm:full-error} proves
\[
 \max_{0\le n\le N}\norm{u_{M}^n-u(t_n)}_{L^2}
 \le C_{T,R,\mu}\bigl(\tau^2+M^{-3}\bigr),
 \qquad \mu>0,
\]
Here $C_{T,R,\mu}$ is independent of $M$ and $N$ and remains
bounded as $\mu\downarrow0$. The step-size threshold is independent of $\mu$,
$M$, and $N$, and no condition couples $\tau$ to $M$. For each fixed $\mu>0$, the
spatial rate improves to the optimal order $M^{-4}$, with a constant that may
depend on $\mu$. For the temporal estimate, the interaction between the heat
flow and convection is represented by commutators. The symmetric splitting
cancels the terms involving the first Lie bracket $[\mu A,F]$, and the remaining terms are controlled in
$L^2$ under the $H^4$ bound without introducing inverse powers of $\mu$. A
dual estimate recovers the optimal spatial order for fixed viscosity. Related
splitting and commutator analyses are given in \cite{BM,HKK,HLR,EO}. The same construction extends to conforming finite elements with a
skew-symmetric convection form, and the resulting stages require only linear solves and
satisfy the discrete energy law. The error analysis is developed for the
Fourier--Galerkin discretization, while the numerical experiments cover both
spatial discretizations, show the expected convergence orders and energy
decay over the tested range of viscosities, and compare the proposed method
with representative schemes.

The rest of this article is organized as follows.
Section~\ref{sec:method} introduces the method and states the main results.
The error analysis is given in Section~\ref{sec:error}.
Section~\ref{sec:finite-element} presents the finite element realization, and
Section~\ref{sec:numerics} contains the numerical experiments.

\section{The numerical method and main results}\label{sec:method}
We first formulate the time integrator and then introduce its Fourier--Galerkin realization, followed by the solvability analysis, energy law, and convergence results.

\subsection{Structure-preserving time integrator}
\label{sec:time-stepping}

Let
\[
 \dot H
 =\{v\in L^2(\T^d)^d:\diver v=0\},
 \qquad \dot H^s=H^s(\T^d)^d\cap\dot H.
\]
Here the dot indicates the divergence-free constraint rather than homogeneous
Sobolev regularity.
The initial velocity in \eqref{eq:NS} is assumed to satisfy
$u_0\in\dot H$.

Let $\PP:L^2(\T^d)^d\to\dot H$ be the $L^2$-orthogonal Leray
projection. Since $\PP$ annihilates gradients and commutes with $\Delta$
on $\T^d$, applying $\PP$ to \eqref{eq:NS} gives
\begin{equation}\label{eq:NS-projected}
 \partial_tu=\mu Au-B(u,u),\qquad u(0)=u_0,
\end{equation}
where
\[
 A=\PP\Delta=\Delta:\dot H^2\to\dot H,
 \qquad B(a,b)=\PP((a\cdot\nabla)b).
\]
The linear flow is
\[
 S(t)=e^{\mu tA}=e^{\mu t\Delta},\qquad t\ge0.
\]

All $L^2$ inner products on $\T^d$ are normalized by its volume.
For divergence-free $a,b,c\in H^3(\T^d)^d$, periodic integration by parts
gives
\begin{equation}\label{eq:skew-continuous}
 \ip{B(a,b)}c=-\ip{B(a,c)}b,
 \qquad \ip{B(a,b)}b=0.
\end{equation}
With $E(v)=\tfrac12\norm v_{L^2}^2$, a sufficiently regular solution of
\eqref{eq:NS-projected} therefore satisfies
\begin{equation}
 \frac{\mathrm d}{\mathrm dt}E(u(t))
 =-\mu\norm{\nabla u(t)}_{L^2}^2 \le 0.
\end{equation}

Let $t_n=n\tau$. We use Strang splitting, with exact heat half-steps surrounding
a convection step of length $\tau$.
The implicit midpoint rule preserves the $L^2$ norm of the convection flow but
requires a nonlinear solve. We instead predict the midpoint advecting velocity
and freeze it in a linear midpoint corrector.

\begin{samepage}
With $u^0=u_0$, the time-semidiscrete step is
\begin{subnumcases}{\label{eq:time-semidiscrete-scheme}}
 w^{n,1}=S(\tau/2)u^n,&\\
 w^{n,2}+\dfrac\tau2 B\bigl(w^{n,1},w^{n,2}\bigr)=w^{n,1},&\\
 w^{n,3}-w^{n,1}
 +\tau B\left(w^{n,2},\dfrac{w^{n,3}+w^{n,1}}{2}\right)=0,&\\
 u^{n+1}=S(\tau/2)w^{n,3}.&
\end{subnumcases}
\end{samepage}
The stages $w^{n,1}$, $w^{n,2}$, and $w^{n,3}$ are the first heat half-step,
the convection midpoint predictor, and the corrected convection endpoint.

Both convection equations are linear in their unknowns. The predictor has
midpoint error $O(\tau^2)$, while \eqref{eq:skew-continuous} makes the corrector
energy-conserving. Dissipation is introduced only by the heat half-steps.

\subsection{Fourier--Galerkin methods}\label{sec:fully-discrete}

For $s\in\R$, we use the normalized Fourier coefficients and the associated
Sobolev norm
\[
 \widehat f(k)=(2\pi)^{-d}\int_{\T^d}f(x)e^{-\mathrm i k\cdot x}\dd x,
 \qquad
 \norm{f}_{H^s}^2=\sum_{k\in\mathbb Z^d}(1+|k|^2)^s|\widehat f(k)|^2.
\]
Let $\Pi_M$ be the $L^2$-orthogonal projection onto
\[
 X_M=\{v:\ \widehat v(k)=0\text{ if }|k|>M,
 \ k\cdot\widehat v(k)=0,
 \ \widehat v(-k)=\overline{\widehat v(k)}\}.
\]
The projection $\Pi_M$ combines Fourier truncation with the Leray projection
and commutes with spatial derivatives and $S(t)$. Standard Fourier
projection estimates \,\cite[Section~2.2.2 and Theorem~2.1]{STW} imply that,
for every divergence-free periodic vector field $v\in H^s$ and
$0\le r\le s$,
\begin{equation}\label{eq:fourier-projection-estimates}
 \norm{\Pi_Mv}_{H^s}\le \norm v_{H^s},
 \qquad
 \norm{(I-\Pi_M)v}_{H^r}\le C M^{r-s}\norm v_{H^s}.
\end{equation}
For divergence-free $a,b$ with
$(a\cdot\nabla)b\in L^2$, set
\begin{equation}\label{eq:BM}
 \BM(a,b)=\Pi_M((a\cdot\nabla)b)=\Pi_MB(a,b).
\end{equation}
The following bounds will be used below.
\begin{lemma}
\label{lem:BM-bounds}
Let $d\in\{2,3\}$, and let $a$ and $b$ be sufficiently regular
divergence-free vector fields, with $b\in X_M$. Then
\begin{align}
 |\ip{\BM(a,b)}b_{H^s}|
 &\le c_s\norm a_{H^s}\norm b_{H^s}^2,
 &&s>d/2+1,\label{eq:sobolev-skew}\\
 \norm{\BM(a,b)}_{L^2}
 &\le\min\bigl\{\norm a_{L^2}\norm{\nabla b}_{L^\infty},
                  \norm a_{L^\infty}\norm{\nabla b}_{L^2}\bigr\},
 \label{eq:basic-prod}\\
 \norm{\BM(a,b)}_{H^2}
 &\le C\norm a_{H^3}\norm b_{H^3}.
 \label{eq:H2prod}
\end{align}
The constants are independent of $M$.
\end{lemma}
\begin{proof}
Let $J^s=(I-\Delta)^{s/2}$. Since $\Pi_M$ is self-adjoint and commutes with
$J^s$, integration by parts gives
\[
 \ip{\BM(a,b)}b_{H^s}
 =\ip{[J^s,a\cdot\nabla]b}{J^s b}_{L^2}.
\]
The Kato--Ponce estimate \cite{KP}, the
periodic dyadic commutator estimate
\cite[Lemma~2.3]{BOZ}, and the embedding
$H^s\hookrightarrow W^{1,\infty}$ yield
\eqref{eq:sobolev-skew}. The $L^2$ bounds follow from the contractivity of
$\Pi_M$ and H\"older's inequality. Finally, the periodic Leibniz estimate
\cite[Proposition~1]{BOZ} and the algebra property of $H^2(\T^d)$ for
$d\le3$ give \eqref{eq:H2prod}.
\end{proof}

Here and below, $M$ denotes the Fourier cutoff. The Galerkin solution satisfies
\begin{equation}\label{eq:NS-Galerkin}
 \partial_tu_M-\mu\Delta u_M+\BM(u_M,u_M)=0,
 \qquad u_M(0)=\Pi_Mu_0.
\end{equation}
Orthogonality of $\Pi_M$ and periodic integration by parts give
\begin{equation}\label{eq:skew}
 \ip{\BM(a,b)}c=-\ip{\BM(a,c)}b,
 \qquad a,b,c\in X_M.
\end{equation}
In particular, $\ip{\BM(a,b)}b=0$.

The Fourier--Galerkin realization of \eqref{eq:time-semidiscrete-scheme} is
initialized by $u_M^0=\Pi_Mu_0$. Given $u_M^n\in X_M$, compute
\begin{subnumcases}{\label{eq:scheme}}
 w_M^{n,1}=S(\tau/2)u_M^n,\label{eq:heat1}&\\
 w_M^{n,2}+\dfrac\tau2\BM\bigl(w_M^{n,1},w_M^{n,2}\bigr)
 =w_M^{n,1},\label{eq:predictor}&\\
 w_M^{n,3}-w_M^{n,1}
 +\tau\BM\left(w_M^{n,2},
       \dfrac{w_M^{n,3}+w_M^{n,1}}{2}\right)=0,
 \label{eq:corrector}&\\
 u_M^{n+1}=S(\tau/2)w_M^{n,3}.\label{eq:heat2}&
\end{subnumcases}
Writing $\BM(a,\cdot)b=\BM(a,b)$, define the predictor and convection maps by
\begin{equation}\label{eq:cayley}
\begin{aligned}
 \mathcal P_\tau(v)
 &:=\left[I+\tfrac\tau2\BM(v,\cdot)\right]^{-1}v,\\
 \mathcal C_\tau(v)
 &:=\left[I+\tfrac\tau2\BM(\mathcal P_\tau(v),\cdot)\right]^{-1}
   \left[I-\tfrac\tau2\BM(\mathcal P_\tau(v),\cdot)\right]v.
\end{aligned}
\end{equation}
The full numerical one-step map is
\begin{equation}\label{eq:numerical-one-step-map}
 \mathcal M_\tau(v)
 =S(\tau/2)\mathcal C_\tau\bigl(S(\tau/2)v\bigr),
 \qquad
 u_M^{n+1}=\mathcal M_\tau(u_M^n).
\end{equation}
For every $t\ge0$ and $s\in\R$, the heat semigroup satisfies
\begin{equation}\label{eq:heat-contraction}
 \widehat{S(t)v}(k)=e^{-\mu t|k|^2}\widehat v(k),
 \qquad
 \norm{S(t)v}_{H^s}\le\norm v_{H^s}.
\end{equation}

\subsection{Discrete energy-decay property}\label{sec:energy-decay}

\begin{theorem}[Discrete energy-decay property]
\label{thm:energy}
For every $M\ge1$, $\tau>0$, and $\mu>0$, each step of
\eqref{eq:scheme} is uniquely solvable in $X_M$.
If the scheme is initialized by $u_M^0=\Pi_Mu_0$, its solution also satisfies
\begin{align}\label{eq:energy-identity}
 E(u_M^{n+1})-E(u_M^n)
 ={}&-\mu\int_0^{\tau/2}\norm{\nabla S(t)u_M^n}_{L^2}^2\dd t\nonumber\\
 &-\mu\int_0^{\tau/2}\norm{\nabla S(t)w_M^{n,3}}_{L^2}^2\dd t.
\end{align}
Consequently,
\begin{align}
 E(u_M^{n+1})\le E(u_M^n),\qquad n\ge0.
\end{align}
\end{theorem}
\begin{proof}
By \eqref{eq:skew}, for $a,b\in X_M$ and $\theta\in\R$,
\begin{equation}
 \norm{b+\theta\BM(a,b)}_{L^2}^2
 =\norm b_{L^2}^2+\theta^2\norm{\BM(a,b)}_{L^2}^2
 \ge\norm b_{L^2}^2.
\end{equation}
Hence $I+\theta\BM(a,\cdot)$ is injective and therefore invertible on the
finite-dimensional space $X_M$, with
\begin{equation}\label{eq:resolvent-contraction}
 \norm{[I+\theta\BM(a,\cdot)]^{-1}f}_{L^2}
 \le\norm f_{L^2}.
\end{equation}
This determines the predictor and corrector uniquely. The heat stages are
well defined because $S(t)$ preserves $X_M$.

Testing \eqref{eq:corrector} with $(w_M^{n,3}+w_M^{n,1})/2$ and using
\eqref{eq:skew} gives
\begin{equation}
 \norm{w_M^{n,3}}_{L^2}=\norm{w_M^{n,1}}_{L^2}.
\end{equation}
For $y(t)=S(t)z$,
\begin{equation}
 \frac{\mathrm d}{\mathrm dt}E(y(t))
 =\mu\ip{\Delta y(t)}{y(t)}
 =-\mu\norm{\nabla y(t)}_{L^2}^2.
\end{equation}
Integrating this identity over the two heat half-steps proves
\eqref{eq:energy-identity}; the energy inequality follows.
\end{proof}

\subsection{Main convergence results}\label{sec:main-errors}

Fix $T>0$ and assume that \eqref{eq:NS} has a strong solution $u$ satisfying
\begin{equation}\label{eq:solution-regularity}
 u\in C([0,T];L^2)\cap L^\infty(0,T;H^4),\qquad
 \sup_{0\le t\le T}\norm{u(t)}_{H^4}\le R,\qquad R\ge1.
\end{equation}
The bound is understood for a strong representative on $[0,T]$. In three
dimensions, it is a conditional assumption on the prescribed interval.
When the estimate is applied to a viscosity-indexed family of
solutions, the bound $R$ is assumed to remain uniform in the
vanishing-viscosity limit.

\begin{theorem}[Error estimate]\label{thm:full-error}
Let $d=2,3$ and assume \eqref{eq:solution-regularity} on $[0,T]$.
For the scheme \eqref{eq:scheme}, initialized by $u_M^0=\Pi_Mu_0$,
for every $\mu>0$ there are constants $C_{T,R,\mu}$ and
$\tau_0(R)>0$ such that
\begin{equation}\label{eq:finite-time-full-error}
 \max_{0\le n\le N}\norm{u_M^n-u(t_n)}_{L^2}
 \le C_{T,R,\mu}\bigl(\tau^2+M^{-3}\bigr),\qquad
 \tau=T/N\le\tau_0(R),\quad N\in\N.
\end{equation}
The constant $C_{T,R,\mu}$ is independent of $M$ and $N$ and
can be chosen so that
\[
 \limsup_{\mu\downarrow0}C_{T,R,\mu}<\infty.
\]
No coupling between $\tau$ and $M$ is required.
For each fixed $\mu>0$, the spatial estimate can be sharpened to
\begin{equation}\label{eq:fixed-viscosity-full-error}
 \max_{0\le n\le N}\norm{u_M^n-u(t_n)}_{L^2}
 \le \widetilde C_{T,R,\mu}
       \bigl(\tau^2+M^{-4}\bigr),
 \qquad \tau=T/N\le\tau_0(R).
\end{equation}
Here $\widetilde C_{T,R,\mu}$ is independent of $M$ and $N$.
\end{theorem}

\section{Error analysis}
\label{sec:error}
We first
derive the equation satisfied by the Fourier projection of the exact solution
and estimate its spatial residual. We then estimate the
consistency defect, including the temporal and spatial contributions. Finally,
a stability estimate gives an exact error recursion, which is closed
by discrete Gronwall.

\subsection{The Fourier projection and the spatial residual}
\label{sec:error-reference}

Following the standard Fourier projection analysis for the periodic
Navier--Stokes equations \cite{Hald81,E93}, set
\begin{equation}\label{eq:projected-reference}
 q_M(t)=\Pi_Mu(t),\qquad q_M^n:=q_M(t_n).
\end{equation}
Introduce the projection and discrete errors
\begin{equation}\label{eq:error-splitting}
 \eta_M(t)=u(t)-q_M(t),\qquad
 \theta_M^n=u_M^n-q_M^n,\qquad
 u_M^n-u(t_n)=\theta_M^n-\eta_M(t_n).
\end{equation}
The $H^4$-stability of $\Pi_M$ and \eqref{eq:solution-regularity} give
\begin{equation}\label{eq:q-H4}
 \sup_{0\le t\le T}\norm{q_M(t)}_{H^4}\le R.
\end{equation}
Since $q_M'=\Pi_Mu'$ and $\Pi_M$ commutes with $A$, applying $\Pi_M$ to
\eqref{eq:NS-projected} and adding and subtracting $B_M(q_M,q_M)$ gives
\begin{equation}\label{eq:projected-equation}
 q_M'=\mu Aq_M-\BM(q_M,q_M)+r_M,\qquad
 r_M:=\BM(q_M,q_M)-\Pi_MB(u,u).
\end{equation}
Thus $r_M$ is the nonlinear projection residual. By bilinearity,
\begin{align}
 r_M
 &=\Pi_M\bigl(B(q_M,q_M)-B(u,u)\bigr)\nonumber\\
 &=-\Pi_MB(q_M,u-q_M)-\Pi_MB(u-q_M,u)\nonumber\\
 &=-\BM(q_M,u-q_M)-\BM(u-q_M,u).
 \label{eq:projection-defect-form}
\end{align}

Under \eqref{eq:solution-regularity}, the Fourier projection estimates
\eqref{eq:fourier-projection-estimates}, with $(s,r)=(4,0)$ and $(4,1)$,
give
\begin{equation}\label{eq:projection-tails}
 \sup_{0\le t\le T}\norm{u(t)-q_M(t)}_{L^2}\le C_RM^{-4},
 \qquad
 \sup_{0\le t\le T}\norm{\nabla(u(t)-q_M(t))}_{L^2}\le C_RM^{-3},
\end{equation}
By \eqref{eq:projection-defect-form}, the $L^2$-contractivity of $\Pi_M$,
Sobolev embedding and \eqref{eq:projection-tails},
\begin{equation}\label{eq:spatial-error}
 \begin{aligned}
 \sup_{0\le t\le T}\norm{r_M(t)}_{L^2}
 &\le \sup_{0\le t\le T}
 \left(
 \norm{q_M}_{L^\infty}\norm{\nabla(u-q_M)}_{L^2}
 +\norm{u-q_M}_{L^2}\norm{\nabla u}_{L^\infty}
 \right)\\
 &\le C_RM^{-3}.
 \end{aligned}
\end{equation}
The same residual has the dual estimate
\begin{equation}\label{eq:spatial-error-dual}
 \sup_{0\le t\le T}\norm{r_M(t)}_{H^{-1}}\le C_RM^{-4}.
\end{equation}
Indeed, for $\phi\in\dot H^1$ and $\phi_M=\Pi_M\phi$,
\[
\begin{aligned}
 |\ip{r_M}{\phi}|
 &\le |\ip{B(q_M,\phi_M)}{u-q_M}|
      +|\ip{B(u-q_M,u)}{\phi_M}|\\
 &\le C_RM^{-4}\norm{\phi}_{H^1}.
\end{aligned}
\]
Here we used \eqref{eq:projection-defect-form}, periodic integration by parts
and \eqref{eq:projection-tails}; taking the supremum over $\phi$ proves
\eqref{eq:spatial-error-dual}. The constants in
\eqref{eq:projection-tails}--\eqref{eq:spatial-error-dual} are independent of
$M$ and $\mu$.

\subsection{Consistency}
\label{sec:error-consistency}

For $v\in X_M$, define the exact Galerkin convection and Navier--Stokes
flows by
\begin{equation}\label{eq:convection-flow}
\begin{aligned}
 \varphi^{\mathrm{conv}}(t,v)&=y(t),
 &y'&=-\BM(y,y),
 &y(0)&=v,\\
 \varphi^{\mathrm{NS}}(t,v)&=z(t),
 &z'&=\mu Az-\BM(z,z),
 &z(0)&=v.
\end{aligned}
\end{equation}
The second line is the flow of \eqref{eq:NS-Galerkin}. For $v\in X_M$, define
the exact heat--convection composition by
\begin{equation}\label{eq:exact-Strang}
 \mathcal S_\tau(v)
 =S(\tau/2)\varphi^{\mathrm{conv}}(\tau,S(\tau/2)v).
\end{equation}

\begin{lemma}
\label{lem:one-step-H4}
Let $R\ge1$, $z\in X_M$ and $\norm{z}_{H^4}\le R$. There is an absolute
constant $c_0>0$ such that, if $0<\tau\le c_0/R$, then
\begin{equation}\label{eq:reference-step-bounds}
\begin{aligned}
 &\sup_{0\le s\le\tau}\norm{\varphi^{\mathrm{NS}}(s,z)}_{H^4}
 +\sup_{0\le r\le t\le\tau}
   \norm{\varphi^{\mathrm{conv}}(r,S(t/2)z)}_{H^4}\\
 &\quad+\norm{\mathcal P_\tau(S(\tau/2)z)}_{H^4}
       +\norm{\mathcal C_\tau(S(\tau/2)z)}_{H^4}
 \le C_R.
\end{aligned}
\end{equation}
The constant $C_R$ is independent of $M$, $\mu$, and $\tau$.
\end{lemma}
\begin{proof}
By \eqref{eq:heat-contraction}, heat propagation is contractive in $H^4$.
For every exact convection or Galerkin trajectory $\zeta$ in
\eqref{eq:reference-step-bounds}, \eqref{eq:sobolev-skew} gives
\begin{equation}\label{eq:reference-flow-H4-bound}
 \frac{\mathrm d}{\mathrm ds}\norm{\zeta(s)}_{H^4}
 \le c\norm{\zeta(s)}_{H^4}^2,
 \qquad
 \norm{\zeta(s)}_{H^4}\le\frac{R}{1-csR}.
\end{equation}
Testing the predictor equation in $H^4$ gives
\begin{equation}\label{eq:predictor-H4-bound}
 \norm{\mathcal P_\tau(S(\tau/2)z)}_{H^4}
 \le
 \frac{\norm{S(\tau/2)z}_{H^4}}
 {1-c\tau\norm{S(\tau/2)z}_{H^4}/2}.
\end{equation}
Testing the corrector equation in $H^4$ gives
\begin{equation}\label{eq:corrector-H4-bound}
\begin{aligned}
 &\norm{\mathcal C_\tau(S(\tau/2)z)}_{H^4}
  -\norm{S(\tau/2)z}_{H^4}\\
 &\quad\le
 \frac{c\tau}{2}\norm{\mathcal P_\tau(S(\tau/2)z)}_{H^4}
 \left(
 \norm{\mathcal C_\tau(S(\tau/2)z)}_{H^4}
 +\norm{S(\tau/2)z}_{H^4}
 \right).
\end{aligned}
\end{equation}
Since $\norm{S(\tau/2)z}_{H^4}\le R$, the result follows for sufficiently
small $c_0$. No inverse inequality is used, and the viscous terms are
nonpositive in the $H^4$ energy estimate. Hence the constants are independent
of $M$ and $\mu>0$.
\end{proof}

Below the lemma is applied with $z=q_M^n$ and the bound in \eqref{eq:q-H4}.
\begin{lemma}[Consistency]
\label{prop:full-step-consistency}
Let $\mu>0$, $t_n+\tau\le T$, and $0<\tau\le c_0/R$ for a sufficiently small absolute
constant $c_0>0$. Then
\begin{equation}\label{eq:full-error}
 \delta_M^{n+1}:=\mathcal M_\tau(q_M^n)-q_M^{n+1},\qquad
 \norm{\delta_M^{n+1}}_{L^2}
 \le C_{R,\mu}\bigl(\tau^3+\tau M^{-3}\bigr).
\end{equation}
The constant $C_{R,\mu}$ is independent of $M$ and $\tau$ and
can be chosen so that
\[
 \limsup_{\mu\downarrow0}C_{R,\mu}<\infty.
\]
\end{lemma}

\begin{proof}
We estimate separately the three terms in the exact decomposition
\begin{align}\label{eq:full-defect-decomposition}
 \delta_M^{n+1}
 ={}&\bigl[\mathcal M_\tau(q_M^n)-\mathcal S_\tau(q_M^n)\bigr]\nonumber\\
 &+\bigl[\mathcal S_\tau(q_M^n)
          -\varphi^{\mathrm{NS}}(\tau,q_M^n)\bigr]\nonumber\\
 &+\bigl[\varphi^{\mathrm{NS}}(\tau,q_M^n)-q_M^{n+1}\bigr].
\end{align}
The three terms are the predictor--corrector defect, the splitting defect and
the spatial residual. Constants below may depend on the dimension, the torus
and $R$, but not on $M$, $\mu$ or $\tau$. We use the standard Sobolev
embeddings in dimensions $d\le3$.
\medskip

\smallskip\noindent\emph{Step 1: Predictor--corrector defect.}
The first term in \eqref{eq:full-defect-decomposition} is the error made by
replacing the exact Galerkin convection flow with the predictor--corrector
map. Let
\begin{equation}\label{eq:pc-reference-flow}
 y(s)=\varphi^{\mathrm{conv}}
 \bigl(s,S(\tau/2)q_M^n\bigr),
 \qquad 0\le s\le\tau.
\end{equation}
Since $y(0)=S(\tau/2)q_M^n$, equations
\eqref{eq:numerical-one-step-map}, \eqref{eq:exact-Strang} and
\eqref{eq:heat-contraction} give
\begin{equation}\label{eq:pc-reduction}
\begin{aligned}
 &\norm{\mathcal M_\tau(q_M^n)-\mathcal S_\tau(q_M^n)}_{L^2}\\
 &\quad\le\norm{\mathcal C_\tau(y(0))-y(\tau)}_{L^2}.
\end{aligned}
\end{equation}
It therefore remains to prove that the last term is bounded by
$C_R\tau^3$.

\smallskip\noindent\emph{Regularity of the exact convection trajectory.}
Equations \eqref{eq:q-H4}, \eqref{eq:heat-contraction} and
Lemma~\ref{lem:one-step-H4} show that the trajectory $y$ is bounded in
$H^4$ by $C_R$ on $[0,\tau]$. Differentiating
\eqref{eq:convection-flow} gives
\begin{equation}\label{eq:convection-time-derivatives}
\begin{aligned}
 y'&=-\BM(y,y),\\
 y''&=-\BM(y',y)-\BM(y,y'),\\
 y'''&=-\BM(y'',y)-2\BM(y',y')-\BM(y,y'').
\end{aligned}
\end{equation}
Applying \eqref{eq:basic-prod}--\eqref{eq:H2prod} therefore gives
\begin{equation}\label{eq:convection-time}
 \sup_s\norm{y'}_{H^2}
 +\sup_s\norm{y''}_{H^1}
 +\sup_s\norm{y'''}_{L^2}\le C_R.
\end{equation}

Subtracting the predictor equation from the integral identity for
$y(\tau/2)$ and using bilinearity gives
\begin{equation}\label{eq:predictor-error-identity}
\begin{aligned}
 &\left[I+\tfrac\tau2\BM(y(0),\cdot)\right]
 \left[\mathcal P_\tau(y(0))-y(\tau/2)\right]\\
 &\quad=\int_0^{\tau/2}
       \BM\bigl(y(s)-y(0),y(\tau/2)\bigr)\dd s\\
 &\qquad+\int_0^{\tau/2}
       \BM\bigl(y(s),y(s)-y(\tau/2)\bigr)\dd s.
\end{aligned}
\end{equation}
The fundamental theorem of calculus and \eqref{eq:convection-time} bound both
increments on the right by $C_R\tau$ in the required norms. Hence
\eqref{eq:basic-prod} and \eqref{eq:resolvent-contraction} give
\begin{equation}\label{eq:predict-local}
 \norm{\mathcal P_\tau(y(0))-y(\tau/2)}_{L^2}
 \le C_R\tau^2.
\end{equation}

Applying midpoint quadrature to
$y(\tau)-y(0)=\int_0^\tau y'(s)\dd s$ and using
$y'(\tau/2)=-\BM(y(\tau/2),y(\tau/2))$ gives
\begin{equation}\label{eq:midpoint-res}
 \norm{y(\tau)-y(0)
 +\tau\BM\bigl(y(\tau/2),y(\tau/2)\bigr)}_{L^2}
 \le\frac{\tau^3}{24}
       \sup_{0\le s\le\tau}\norm{y'''(s)}_{L^2}
 \le C_R\tau^3.
\end{equation}
Taylor expansion about $\tau/2$ also gives the endpoint-average estimate
\begin{equation}\label{eq:average-res}
 \norm{\frac{y(\tau)+y(0)}2-y(\tau/2)}_{H^1}
 \le\frac{\tau^2}{8}
       \sup_{0\le s\le\tau}\norm{y''(s)}_{H^1}
 \le C_R\tau^2.
\end{equation}

In the corrector equation, decompose its endpoint average as
\[
 \frac{\mathcal C_\tau(y(0))+y(0)}2
 =\frac{\mathcal C_\tau(y(0))-y(\tau)}2
  +\frac{y(\tau)+y(0)}2.
\]
Adding and subtracting $\tau\BM(y(\tau/2),y(\tau/2))$ gives
\begin{equation}\label{eq:corrector-error-identity}
\begin{aligned}
 &\left[I+\tfrac\tau2\BM(\mathcal P_\tau(y(0)),\cdot)\right]
 \left[\mathcal C_\tau(y(0))-y(\tau)\right]\\
 &=-\left[y(\tau)-y(0)
      +\tau\BM\bigl(y(\tau/2),y(\tau/2)\bigr)\right]\\
 &\quad-\tau\BM\left(
       \mathcal P_\tau(y(0))-y(\tau/2),
       \frac{y(\tau)+y(0)}2
       \right)\\
 &\quad-\tau\BM\left(
       y(\tau/2),
       \frac{y(\tau)+y(0)}2-y(\tau/2)
       \right).
\end{aligned}
\end{equation}
Applying \eqref{eq:basic-prod} and \eqref{eq:resolvent-contraction} to
\eqref{eq:corrector-error-identity}, and then using
\eqref{eq:predict-local}, \eqref{eq:midpoint-res} and
\eqref{eq:average-res}, gives
\begin{equation}\label{eq:convection-map-defect}
\begin{aligned}
 \norm{\mathcal C_\tau(y(0))-y(\tau)}_{L^2}
 &\le C_R\Biggl(\tau^3
       +\tau\norm{\mathcal P_\tau(y(0))-y(\tau/2)}_{L^2}\\
 &\hspace{11mm}
       +\tau\norm{\frac{y(\tau)+y(0)}2-y(\tau/2)}_{H^1}\Biggr)
 \le C_R\tau^3.
\end{aligned}
\end{equation}
Combining \eqref{eq:pc-reduction} and
\eqref{eq:convection-map-defect} gives
\begin{equation}\label{eq:predictor-corrector-part}
 \norm{\mathcal M_\tau(q_M^n)-\mathcal S_\tau(q_M^n)}_{L^2}
 \le C_R\tau^3.
\end{equation}

\smallskip\noindent\emph{Step 2: Local error of the exact Strang splitting.}
Set $v=q_M^n$. Only in this step, write
$F(u)=-\BM(u,u)$, so that
\begin{equation}\label{eq:convection-vector-field-derivative}
 F'(u)\eta=-\BM(\eta,u)-\BM(u,\eta).
\end{equation}
For $0\le t\le\tau$, regard $\mathcal S_t(v)$ as a path and define its
defect with respect to the Galerkin Navier--Stokes equation by
\begin{equation}\label{eq:Strang-defect-definition}
 \mathcal R(t)
 =\frac{\mathrm d}{\mathrm dt}\mathcal S_t(v)
  -\mu A\mathcal S_t(v)-F(\mathcal S_t(v)).
\end{equation}
For fixed $t$, let
\begin{equation}\label{eq:time-dependent-convection-path}
 y_t(r)=\varphi^{\mathrm{conv}}\bigl(r,S(t/2)v\bigr),
 \qquad 0\le r\le t.
\end{equation}
Then $\mathcal S_t(v)=S(t/2)y_t(t)$. Differentiation with respect to
$t$ gives
\begin{equation}\label{eq:Strang-path-derivative}
\begin{aligned}
 \frac{\mathrm d}{\mathrm dt}\mathcal S_t(v)
 ={}&\frac{\mu}{2}A\mathcal S_t(v)+S(t/2)F(y_t(t))\\
 &+\frac12S(t/2)
 \mathrm D\varphi^{\mathrm{conv}}\bigl(t,S(t/2)v\bigr)
 \mu AS(t/2)v.
\end{aligned}
\end{equation}
Since $A$ commutes with $S(t/2)$, the defect therefore has the form
\begin{equation}\label{eq:raw-Strang-defect}
\begin{aligned}
 \mathcal R(t)
 ={}&S(t/2)F(y_t(t))-F(\mathcal S_t(v))\\
 &+\frac12S(t/2)\Bigl[
 \mathrm D\varphi^{\mathrm{conv}}\bigl(t,S(t/2)v\bigr)
 \mu AS(t/2)v-\mu Ay_t(t)\Bigr].
\end{aligned}
\end{equation}

For vector fields $G$ and $H$ on $X_M$, write
\begin{equation}\label{eq:Lie-bracket-definition}
 [G,H](u)=G'(u)H(u)-H'(u)G(u),
\end{equation}
and set
\begin{equation}\label{eq:first-commutator}
 \Gamma(u)=[\mu A,F](u)
 =\mu\bigl(AF(u)-F'(u)Au\bigr).
\end{equation}
The commutator representations in \cite[Lemma~5]{AHKT}, applied to
$\mu A$ and $F$ in $X_M$, give
\begin{equation}\label{eq:first-commutator-identities}
\begin{aligned}
 S(t/2)F(y_t(t))-F(S(t/2)y_t(t))
 &=\int_0^{t/2}S(t/2-r)\Gamma(S(r)y_t(t))\dd r,\\
 \mathrm D\varphi^{\mathrm{conv}}\bigl(t,S(t/2)v\bigr)
 \mu AS(t/2)v-\mu Ay_t(t)
 &=-\int_0^t
 \mathrm D\varphi^{\mathrm{conv}}(t-r,y_t(r))
 \Gamma(y_t(r))\dd r.
\end{aligned}
\end{equation}
Substitution into \eqref{eq:raw-Strang-defect} yields
\begin{equation}\label{eq:Strang-defect-first-representation}
\begin{aligned}
 \mathcal R(t)
 ={}&\int_0^{t/2}S(t/2-r)\Gamma(S(r)y_t(t))\dd r\\
 &-\frac12S(t/2)\int_0^t
 \mathrm D\varphi^{\mathrm{conv}}(t-r,y_t(r))
 \Gamma(y_t(r))\dd r.
\end{aligned}
\end{equation}
Direct differentiation, using $y_t'=F(y_t)$ and the variational equation for
$\varphi^{\mathrm{conv}}$, gives
\begin{equation}\label{eq:transported-commutator-identities}
\begin{aligned}
 \frac{\mathrm d}{\mathrm dr}
 \bigl[S(t/2-r)\Gamma(S(r)y_t(t))\bigr]
 &=-S(t/2-r)[\mu A,\Gamma](S(r)y_t(t)),\\
 \frac{\mathrm d}{\mathrm dr}
 \bigl[\mathrm D\varphi^{\mathrm{conv}}(t-r,y_t(r))\Gamma(y_t(r))\bigr]
 &=-\mathrm D\varphi^{\mathrm{conv}}(t-r,y_t(r))
 [F,\Gamma](y_t(r)).
\end{aligned}
\end{equation}
Integration by parts in the two integrals in
\eqref{eq:Strang-defect-first-representation} produces the boundary terms
$\tfrac t2S(t/2)\Gamma(y_t(t))$ and
$-\tfrac t2S(t/2)\Gamma(y_t(t))$, which cancel. The remaining terms give
\begin{align}\label{eq:defect-final}
 \mathcal R(t)
 ={}&-\int_0^{t/2}(t/2-r)S(t/2-r)
 [\mu A,\Gamma](S(r)y_t(t))\dd r\nonumber\\
 &-\frac12S(t/2)\int_0^t r\,
 \mathrm D\varphi^{\mathrm{conv}}(t-r,y_t(r))
 [F,\Gamma](y_t(r))\dd r.
\end{align}

Using the $H^4$ reference bounds in Lemma~\ref{lem:one-step-H4}, the product
rule and commutation of $\Pi_M$ with derivatives give
\begin{equation}\label{eq:commutator-identities}
 \Gamma(u)=-2\mu\sum_j\BM(\partial_ju,\partial_ju),
 \qquad
 [\mu A,\Gamma](u)
 =-4\mu^2\sum_{i,j}\BM(\partial_{ij}u,\partial_{ij}u).
\end{equation}
For $d\le3$ and $\norm u_{H^4}\le C_R$, the embedding
$H^2\hookrightarrow L^\infty$ gives
\begin{equation}\label{eq:second-commutator-bound}
 \norm{[\mu A,\Gamma](u)}_{L^2}
 \le C\mu^2\norm{\nabla^2u}_{L^\infty}
              \norm{\nabla^3u}_{L^2}
 \le C_R\mu^2.
\end{equation}
Moreover,
\begin{equation}\label{eq:first-commutator-derivative}
 \Gamma'(u)\eta=-2\mu\sum_j
 \bigl[\BM(\partial_j\eta,\partial_ju)
       +\BM(\partial_ju,\partial_j\eta)\bigr].
\end{equation}
The product estimates imply
\begin{equation}\label{eq:vector-field-commutator-bounds}
 \norm{F(u)}_{H^3}\le C_R,\qquad
 \norm{\Gamma(u)}_{H^2}\le C_R\mu,
\end{equation}
and hence
\begin{equation}\label{eq:mixed-commutator-terms}
 \norm{F'(u)\Gamma(u)}_{L^2}
 +\norm{\Gamma'(u)F(u)}_{L^2}
 \le C_R\mu.
\end{equation}
By the definition of the Lie bracket,
\begin{equation}\label{eq:nonlinear-commutator-bound}
 \norm{[F,\Gamma](u)}_{L^2}\le C_R\mu.
\end{equation}

The reference path is bounded in $H^3$ by
Lemma~\ref{lem:one-step-H4}. Its variational equation and
skew-symmetry give
\begin{equation}\label{eq:tangent}
 \norm{\mathrm D\varphi^{\mathrm{conv}}
       (\ell,y_t(r))\eta}_{L^2}
 \le e^{C_R\ell}\norm{\eta}_{L^2},
 \qquad 0\le r\le r+\ell\le t.
\end{equation}
All flow times in \eqref{eq:defect-final} are nonnegative. The preceding
commutator bounds, \eqref{eq:defect-final}, and \eqref{eq:tangent} imply
\begin{equation}\label{eq:strong-defect}
 \norm{\mathcal R(t)}_{L^2}\le
 C_R(\mu+\mu^2)t^2,
 \qquad \mu>0.
\end{equation}

Finally, set
\begin{equation}\label{eq:Strang-error-definition}
 e(t)=\mathcal S_t(v)-\varphi^{\mathrm{NS}}(t,v).
\end{equation}
Then
\begin{equation}\label{eq:Strang-error-equation}
\begin{aligned}
 e'={}&\mu Ae
 +F(\mathcal S_t(v))-F(\varphi^{\mathrm{NS}}(t,v))
 +\mathcal R(t),\\
 e(0)={}&0.
\end{aligned}
\end{equation}
Since
\begin{equation}\label{eq:Strang-nonlinear-difference}
\begin{aligned}
 F(\mathcal S_t(v))-F(\varphi^{\mathrm{NS}}(t,v))
 ={}&-\BM(\mathcal S_t(v),e)\\
 &-\BM(e,\varphi^{\mathrm{NS}}(t,v)),
\end{aligned}
\end{equation}
skew-symmetry, Lemma~\ref{lem:one-step-H4}, and
\eqref{eq:strong-defect} give
\begin{equation}\label{eq:Strang-error-energy}
 \frac12\frac{\mathrm d}{\mathrm dt}\norm e_{L^2}^2
 +\mu\norm{\nabla e}_{L^2}^2
 \le C_R\norm e_{L^2}^2
 +C_R(\mu+\mu^2)t^2\norm e_{L^2}.
\end{equation}
Applying Gronwall's inequality after the standard regularization
$y_\varepsilon(t)=(\norm{e(t)}_{L^2}^2+\varepsilon)^{1/2}$ and then letting
$\varepsilon\downarrow0$ gives
\begin{equation}\label{eq:Strang-error-Gronwall}
 \norm{e(t)}_{L^2}
 \le C_R(\mu+\mu^2)
       \int_0^t e^{C_R(t-s)}s^2\dd s
 \le C_R(\mu+\mu^2)t^3.
\end{equation}
Taking $t=\tau$, we obtain
\begin{equation}\label{eq:splitting-part}
 \norm{\mathcal S_\tau(v)-\varphi^{\mathrm{NS}}(\tau,v)}_{L^2}
 \le C_R(\mu+\mu^2)\tau^3,
 \qquad \mu>0.
\end{equation}

\medskip

\smallskip\noindent\emph{Step 3: Contribution of the projection residual.}
The first path below solves the Galerkin equation with $r_M=0$, while the
second solves \eqref{eq:projected-equation} with the same initial value.
Thus $d_M(s):=\varphi^{\mathrm{NS}}(s,q_M^n)-q_M(t_n+s)$ satisfies
\begin{equation}\label{eq:projection-path-error-equation}
 \begin{aligned}
 d_M'={}&\mu Ad_M
 -\BM(\varphi^{\mathrm{NS}}(s,q_M^n),d_M)
 -\BM(d_M,q_M(t_n+s))-r_M(t_n+s),\\
 d_M(0)={}&0.
 \end{aligned}
\end{equation}
Testing with $d_M$, the first convection term vanishes by skew-symmetry,
while the second is bounded by
$\norm{\nabla q_M(t_n+s)}_{L^\infty}\norm{d_M}_{L^2}^2
\le C_R\norm{d_M}_{L^2}^2$ in view of \eqref{eq:q-H4}. Hence
\begin{equation}\label{eq:projection-path-energy}
 \frac12\frac{\mathrm d}{\mathrm ds}\norm{d_M}_{L^2}^2
 +\mu\norm{\nabla d_M}_{L^2}^2
 \le C_R\norm{d_M}_{L^2}^2
 +\norm{r_M(t_n+s)}_{L^2}\norm{d_M}_{L^2}.
\end{equation}
The same regularized Gronwall argument, now applied to
$y(s)=\norm{d_M(s)}_{L^2}$ with $y(0)=0$, yields
\begin{equation}\label{eq:projection-path-Gronwall}
 \norm{d_M(s)}_{L^2}
 \le \int_0^s e^{C_R(s-r)}
 \norm{r_M(t_n+r)}_{L^2}\dd r,
 \qquad 0\le s\le\tau.
\end{equation}
Taking $s=\tau$ and using \eqref{eq:spatial-error} shows that accumulation
over one time interval supplies the factor $\tau$:
\begin{equation}\label{eq:projected-reference-part}
\begin{aligned}
 \norm{\varphi^{\mathrm{NS}}(\tau,q_M^n)-q_M^{n+1}}_{L^2}
 &\le e^{C_R\tau}\int_{t_n}^{t_{n+1}}\norm{r_M(t)}_{L^2}\dd t\\
 &\le C_R\tau M^{-3}.
\end{aligned}
\end{equation}

Combining \eqref{eq:predictor-corrector-part},
\eqref{eq:splitting-part} and \eqref{eq:projected-reference-part} in
\eqref{eq:full-defect-decomposition} gives
\[
 \norm{\delta_M^{n+1}}_{L^2}
 \le C_R\left[(1+\mu+\mu^2)\tau^3+\tau M^{-3}\right]
 \le C_{R,\mu}\bigl(\tau^3+\tau M^{-3}\bigr).
\]
Since $1+\mu+\mu^2$ remains bounded as $\mu\downarrow0$,
the constant $C_{R,\mu}$ has the property stated in
Lemma~\ref{prop:full-step-consistency}.
\end{proof}

\medskip

\begin{lemma}
\label{lem:fixed-viscosity-projection-defect}
For fixed $\mu>0$, set
\begin{equation}\label{eq:fixed-spatial-defect-definition}
 \delta_{x,M}^{n+1}
 =\varphi^{\mathrm{NS}}(\tau,q_M^n)-q_M^{n+1}.
\end{equation}
Under the assumptions of Lemma~\ref{prop:full-step-consistency},
\begin{equation}\label{eq:fixed-viscosity-projection-defect}
 \norm{\delta_{x,M}^{n+1}}_{H^{-1}}
 \le \widetilde C_{R,\mu}\tau M^{-4},
 \qquad
 \norm{\delta_{x,M}^{n+1}}_{L^2}^2
 \le \widetilde C_{R,\mu}\tau M^{-8}.
\end{equation}
\end{lemma}
\begin{proof}
Let $d_M$ be the path introduced in Step~3 of the consistency proof. Its
mean is zero. Set $\psi_M=(-A)^{-1}d_M$ and
$y=\norm{\nabla\psi_M}_{L^2}\simeq\norm{d_M}_{H^{-1}}$. Testing the equation
for $d_M$ first with $\psi_M$ and then with $d_M$, and using
Lemma~\ref{lem:one-step-H4} and \eqref{eq:q-H4}, gives
\begin{equation}\label{eq:fixed-projection-defect-energy}
\begin{aligned}
 \frac12\frac{\mathrm d}{\mathrm ds}y^2
 +\mu\norm{d_M}_{L^2}^2
 &\le C_R\norm{d_M}_{L^2}y
   +\norm{r_M(t_n+s)}_{H^{-1}}y,\\
 \frac12\frac{\mathrm d}{\mathrm ds}\norm{d_M}_{L^2}^2
 +\mu\norm{\nabla d_M}_{L^2}^2
 &\le C_R\norm{d_M}_{L^2}^2
   +\norm{r_M(t_n+s)}_{H^{-1}}\norm{d_M}_{H^1}.
\end{aligned}
\end{equation}
The Poincar\'e inequality, Young's inequality, and
\eqref{eq:spatial-error-dual} imply
\begin{equation}\label{eq:fixed-projection-defect-Young}
\begin{aligned}
 \frac12\frac{\mathrm d}{\mathrm ds}y^2
 +\frac\mu2\norm{d_M}_{L^2}^2
 &\le \widetilde C_{R,\mu}y^2+C_RM^{-4}y,\\
 \frac12\frac{\mathrm d}{\mathrm ds}\norm{d_M}_{L^2}^2
 +\frac\mu2\norm{\nabla d_M}_{L^2}^2
 &\le C_R\norm{d_M}_{L^2}^2
      +\widetilde C_{R,\mu}M^{-8}.
\end{aligned}
\end{equation}
Since $d_M(0)=0$, the same regularized Gronwall argument
applied to the first line, together with the usual Gronwall inequality for the
second, gives
\begin{equation}\label{eq:fixed-projection-defect-path-bounds}
 \norm{d_M(s)}_{H^{-1}}\le
 \widetilde C_{R,\mu}sM^{-4},
 \qquad
 \norm{d_M(s)}_{L^2}^2\le
 \widetilde C_{R,\mu}sM^{-8}.
\end{equation}
Taking $s=\tau$ proves \eqref{eq:fixed-viscosity-projection-defect}.
\end{proof}

\medskip

\subsection{Stability and convergence}
\label{sec:error-stability}

\begin{lemma}[Stability]\label{lem:stability}
Let $v_1,v_2\in X_M$ satisfy $\norm{v_2}_{H^4}\le R$ and
$\tau\le c_0/R$. Then
\begin{equation}\label{eq:fullstability}
 \norm{\mathcal M_\tau(v_1)-\mathcal M_\tau(v_2)}_{L^2}
 \le(1+C_R\tau)\norm{v_1-v_2}_{L^2}.
\end{equation}
More precisely,
\begin{equation}\label{eq:dissipative-fullstability}
\begin{aligned}
 &\norm{\mathcal M_\tau(v_1)-\mathcal M_\tau(v_2)}_{L^2}^2
 +\mu\tau\norm{\nabla\bigl(\mathcal M_\tau(v_1)
                  -\mathcal M_\tau(v_2)\bigr)}_{L^2}^2\\
 &\qquad\le(1+C_R\tau)\norm{v_1-v_2}_{L^2}^2.
\end{aligned}
\end{equation}
\end{lemma}
\begin{proof}
Set $\xi=S(\tau/2)v_1$ and $\zeta=S(\tau/2)v_2$. For
$\alpha\in\{\xi,\zeta\}$, let $w_\alpha^1=\alpha$ and denote the
corresponding predictor and corrector stages by $w_\alpha^2$ and
$w_\alpha^3$. With $e_j=w_\xi^j-w_\zeta^j$, subtraction of the stage
equations gives
\[
\begin{aligned}
 e_2+\tfrac\tau2\BM(\xi,e_2)
 &=e_1-\tfrac\tau2\BM(e_1,w_\zeta^2),\\
 e_3-e_1+\tau\BM(w_\xi^2,(e_3+e_1)/2)
 &=-\tau\BM(e_2,(w_\zeta^3+w_\zeta^1)/2).
\end{aligned}
\]
The resolvent contraction and \eqref{eq:basic-prod} yield
\[
 \norm{e_2}_{L^2}
 \le\left(1+\tfrac\tau2\norm{\nabla w_\zeta^2}_{L^\infty}\right)
       \norm{e_1}_{L^2}.
\]
Taking the $L^2$ inner product of the corrector equation with
$(e_3+e_1)/2$ and using \eqref{eq:skew} gives
\[
\begin{aligned}
 \frac12\bigl(\norm{e_3}_{L^2}^2-\norm{e_1}_{L^2}^2\bigr)
 \le{}&\frac\tau2
 \norm{\nabla((w_\zeta^3+w_\zeta^1)/2)}_{L^\infty}
 \norm{e_2}_{L^2}
 \bigl(\norm{e_3}_{L^2}+\norm{e_1}_{L^2}\bigr).
\end{aligned}
\]
Factoring the left-hand side and cancelling the common nonnegative factor,
with the zero case being immediate, we obtain
\[
\begin{aligned}
 \norm{e_3}_{L^2}
 \le\biggl[1+\tau\norm{\nabla((w_\zeta^3+w_\zeta^1)/2)}_{L^\infty}
 \left(1+\tfrac\tau2\norm{\nabla w_\zeta^2}_{L^\infty}\right)\biggr]
 \norm{e_1}_{L^2}.
\end{aligned}
\]
By \eqref{eq:heat-contraction}, Lemma~\ref{lem:one-step-H4}, and
$H^4\hookrightarrow W^{1,\infty}$,
\[
 \norm{e_1}_{L^2}\le\norm{v_1-v_2}_{L^2},
 \qquad
 \norm{\nabla w_\zeta^2}_{L^\infty}
 +\norm{\nabla((w_\zeta^3+w_\zeta^1)/2)}_{L^\infty}
 \le C_R.
\]
The step-size condition absorbs the quadratic term in $\tau$. The final heat
half-step is an $L^2$ contraction, which proves \eqref{eq:fullstability}.

For the stronger estimate, $(1+x)e^{-x}\le1$ for $x\ge0$ and the Fourier
representation of $S(\tau/2)$ give
\[
 \norm{S(\tau/2)v}_{L^2}^2
 +\mu\tau\norm{\nabla S(\tau/2)v}_{L^2}^2
 \le\norm v_{L^2}^2.
\]
Apply this inequality to
$v=\mathcal C_\tau(\xi)-\mathcal C_\tau(\zeta)$ and square the convection
stability estimate above. After reducing $c_0$ if necessary, its quadratic
factor in $\tau$ is bounded by $1+C_R\tau$, proving
\eqref{eq:dissipative-fullstability}.
\end{proof}

\medskip

\begin{proof}[Proof of Theorem~\ref{thm:full-error}]
Lemma~\ref{prop:full-step-consistency} gives the complete local defect
with respect to the Fourier projection $q_M(t)=\Pi_Mu(t)$. By
\eqref{eq:error-splitting}, the discrete error satisfies
\[
 \theta_M^{n+1}
 =\mathcal M_\tau(u_M^n)-\mathcal M_\tau(q_M^n)+\delta_M^{n+1}.
\]
Since \eqref{eq:q-H4} controls the reference stages uniformly at every time
level, taking $L^2$ norms and applying \eqref{eq:fullstability} and
\eqref{eq:full-error} give
\begin{equation}\label{eq:uniform-error-recursion}
 \norm{\theta_M^{n+1}}_{L^2}
 \le(1+C_R\tau)\norm{\theta_M^n}_{L^2}
      +C_{R,\mu}\bigl(\tau^3+\tau M^{-3}\bigr).
\end{equation}
Since $\theta_M^0=0$, discrete Gronwall, \eqref{eq:error-splitting}, and
\eqref{eq:projection-tails} yield
\begin{align}\label{eq:uniform-full-error-conclusion}
 \max_{0\le n\le N}\norm{\theta_M^n}_{L^2}
 &\le C_{T,R,\mu}\bigl(\tau^2+M^{-3}\bigr),\nonumber\\
 \max_{0\le n\le N}\norm{u_M^n-u(t_n)}_{L^2}
 &\le C_{T,R,\mu}\bigl(\tau^2+M^{-3}\bigr).
\end{align}
All step restrictions are covered by $\tau\le\tau_0(R):=c_0/R$. The constant
$C_{T,R,\mu}$ is independent of $M$ and $N$ and remains
bounded as $\mu\downarrow0$, because the viscosity enters the temporal
consistency bound only through nonnegative powers. The step-size threshold is
independent of $\mu$, $M$, and $N$. No global high-order bound for the
numerical or Galerkin solution is assumed.
For the fixed-viscosity estimate, set
\begin{equation}\label{eq:fixed-defect-decomposition}
\begin{aligned}
 \delta_{\tau,M}^{n+1}
 &=\mathcal M_\tau(q_M^n)
   -\varphi^{\mathrm{NS}}(\tau,q_M^n),\\
 \delta_{x,M}^{n+1}
 &=\varphi^{\mathrm{NS}}(\tau,q_M^n)-q_M^{n+1},\\
 a_M^{n+1}
 &=\mathcal M_\tau(u_M^n)-\mathcal M_\tau(q_M^n).
\end{aligned}
\end{equation}
Steps~1 and 2 of the consistency proof,
Lemma~\ref{lem:fixed-viscosity-projection-defect}, and
\eqref{eq:dissipative-fullstability} give, respectively,
\begin{align}
 \norm{\delta_{\tau,M}^{n+1}}_{L^2}
 &\le C_{R,\mu}\tau^3,\nonumber\\
 \norm{\delta_{x,M}^{n+1}}_{H^{-1}}
 &\le \widetilde C_{R,\mu}\tau M^{-4},
 &
 \norm{\delta_{x,M}^{n+1}}_{L^2}^2
 &\le \widetilde C_{R,\mu}\tau M^{-8},\nonumber\\
 \norm{a_M^{n+1}}_{L^2}^2
 +\mu\tau\norm{\nabla a_M^{n+1}}_{L^2}^2
 &\le(1+C_R\tau)\norm{\theta_M^n}_{L^2}^2.
 \label{eq:fixed-viscosity-input-bounds}
\end{align}
The error identity
$\theta_M^{n+1}=a_M^{n+1}+\delta_{\tau,M}^{n+1}
+\delta_{x,M}^{n+1}$ and Young's inequality imply
\begin{equation}\label{eq:fixed-defect-pairings}
\begin{aligned}
 2|\ip{a_M^{n+1}}{\delta_{\tau,M}^{n+1}}|
 &\le\tau\norm{a_M^{n+1}}_{L^2}^2
      +C_{R,\mu}\tau^5,\\
 2|\ip{a_M^{n+1}}{\delta_{x,M}^{n+1}}|
 &\le\tau\norm{a_M^{n+1}}_{L^2}^2
 +\frac{\mu\tau}{2}\norm{\nabla a_M^{n+1}}_{L^2}^2
 +\widetilde C_{R,\mu}\tau M^{-8}.
\end{aligned}
\end{equation}
Expanding the square, using \eqref{eq:fixed-viscosity-input-bounds}, and
choosing $c_0\le1$ in $\tau_0(R)=c_0/R$, we obtain
\begin{equation}\label{eq:fixed-viscosity-error-recursion}
 \norm{\theta_M^{n+1}}_{L^2}^2
 \le(1+C_R\tau)\norm{\theta_M^n}_{L^2}^2
     +\widetilde C_{R,\mu}
       \bigl(\tau^5+\tau M^{-8}\bigr).
\end{equation}

Since $\theta_M^0=0$, discrete Gronwall, followed by taking square roots and
using \eqref{eq:projection-tails}, gives
\begin{align}\label{eq:fixed-viscosity-error-conclusion}
 \max_{0\le n\le N}\norm{\theta_M^n}_{L^2}^2
 &\le \widetilde C_{T,R,\mu}
       \bigl(\tau^4+M^{-8}\bigr),\nonumber\\
 \max_{0\le n\le N}\norm{u_M^n-u(t_n)}_{L^2}
 &\le \widetilde C_{T,R,\mu}
       \bigl(\tau^2+M^{-4}\bigr).
\end{align}
\end{proof}

\section{A finite element method}\label{sec:finite-element}

The time integrator in Subsection~\ref{sec:time-stepping} also admits the
following conforming finite element realization.

Let $X_h\subset H^1(\T^d)^d$ and $Q_h\subset L_0^2(\T^d)$ be a conforming,
inf--sup stable periodic velocity--pressure pair with $h$ denoting the finite element mesh size, and set
\begin{equation}\label{eq:discrete-solenoidal-space}
 V_h
 =\{v_h\in X_h:\ \ip{\diver v_h}{q_h}=0
                  \text{ for every }q_h\in Q_h\}.
\end{equation}
Let $P_h$ be the $L^2$-orthogonal projection onto
$V_h$. Define the discrete Stokes operator
$A_h:V_h\to V_h$ by
\begin{equation}\label{eq:Ah}
 \ip{A_hv_h}{z_h}=-\ip{\nabla v_h}{\nabla z_h},
 \qquad v_h,z_h\in V_h,
\end{equation}
and set $S_h(t)=e^{\mu tA_h}$.

For $a_h,b_h\in V_h$, define $B_h(a_h,b_h)\in V_h$ by
\begin{equation}\label{eq:Bh}
 \ip{B_h(a_h,b_h)}{z_h}
 =\frac12\ip{(a_h\cdot\nabla)b_h}{z_h}
  -\frac12\ip{(a_h\cdot\nabla)z_h}{b_h},
 \qquad z_h\in V_h.
\end{equation}
Thus $\ip{B_h(a_h,b_h)}{b_h}=0$ without requiring pointwise
divergence-free velocities.

Initialize $u_h^0=P_h u_0$. Given $u_h^n\in V_h$, define
\begin{subnumcases}{\label{eq:fe-scheme}}
 w_h^{n,1}=S_h(\tau/2)u_h^n,&\\
 w_h^{n,2}+\dfrac\tau2 B_h\bigl(w_h^{n,1},w_h^{n,2}\bigr)=w_h^{n,1},&\\
 w_h^{n,3}-w_h^{n,1}
 +\tau B_h\left(w_h^{n,2},
       \dfrac{w_h^{n,3}+w_h^{n,1}}{2}\right)=0,&\\
 u_h^{n+1}=S_h(\tau/2)w_h^{n,3}.&
\end{subnumcases}


\begin{lemma}[Discrete energy decay property]
\label{prop:fe-structure}
For every $\tau>0$ and $\mu>0$, the finite element step
\eqref{eq:fe-scheme} is uniquely solvable. Moreover,
\begin{align}\label{eq:fe-energy}
 E(u_h^{n+1})-E(u_h^n)
 ={}&-\mu\int_0^{\tau/2}
       \norm{\nabla S_h(t)u_h^n}_{L^2}^2\dd t\nonumber\\
 &-\mu\int_0^{\tau/2}
       \norm{\nabla S_h(t)w_h^{n,3}}_{L^2}^2\dd t
 \le0.
\end{align}
Consequently,
\begin{align}
 E(u_h^{n+1})\le E(u_h^n),\qquad n\ge0.
\end{align}
\end{lemma}
\begin{proof}
For $a_h,b_h\in V_h$ and $\theta\in\R$, \eqref{eq:Bh} gives
\[
 \norm{b_h+\theta B_h(a_h,b_h)}_{L^2}^2
 =\norm{b_h}_{L^2}^2+\theta^2\norm{B_h(a_h,b_h)}_{L^2}^2.
\]
Hence $I+\theta B_h(a_h,\cdot)$ is invertible on the finite-dimensional
space $V_h$, so both convection stages are uniquely solvable.
The heat stages are well defined by $S_h(t)$. 

Testing
the corrector with $(w_h^{n,3}+w_h^{n,1})/2$ gives
$E(w_h^{n,3})=E(w_h^{n,1})$, while
\eqref{eq:Ah} gives, for $q_h(t)=S_h(t)z_h$,
\[
 \frac{\mathrm d}{\mathrm dt}E(q_h(t))
 =\mu\ip{A_hq_h(t)}{q_h(t)}
 =-\mu\norm{\nabla q_h(t)}_{L^2}^2.
\]
Integrating over the two heat half-steps proves \eqref{eq:fe-energy}.
\end{proof}

\section{Numerical experiments}\label{sec:numerics}

We present numerical results for the convergence and energy behavior of the
proposed method. All computations are two-dimensional on $[0,2\pi]^2$. we use the initial data
\begin{align}\label{eq:numerical-h4-data}
 \psi_m(x_1,x_2)& =\sin^m(x_1/2)\sin^m(x_2/2) \quad u_0 =(\partial_2\psi_m,-\partial_1\psi_m),\quad m=4.6.
\end{align}
The resulting velocity satisfies $u_0\in H^{4+\varepsilon}(\mathbb T^2)$ for
$0\le\varepsilon<0.1$. The Fourier approximation uses the circular cutoff
$|k|\le M$ and a $(4M+1)$-point FFT grid used to evaluate the
projected quadratic convolution without aliasing. For the fully discrete solution $u_{\tau,M}^n$, we
measure temporal and spatial convergence by
\begin{subequations}\label{eq:numerical-self-convergence}
\begin{align}
 D_\tau^{(M)}
 &=\max_{0\le n\le T/\tau}
   \norm{u_{\tau,M}^n-u_{\tau/2,M}^{2n}}_{L^2},
 &p_\tau
 &=\log_2\frac{D_\tau^{(M)}}{D_{\tau/2}^{(M)}},\label{eq:temporal-self-convergence}\\
 D_M^{(\tau)}
 &=\max_{0\le n\le N}
   \norm{\iota_M^{2M}u_{\tau,M}^{n}-u_{\tau,2M}^{n}}_{L^2},
 &p_M
 &=\log_2\frac{D_M^{(\tau)}}{D_{2M}^{(\tau)}}.
 \label{eq:spatial-self-convergence}
\end{align}
\end{subequations}
Here $\iota_M^{2M}$ is the canonical embedding from $X_M$ into $X_{2M}$.

We first examine the temporal and spatial convergence for six viscosities ranging from $\mu=1$ to $10^{-8}$. Figure~\ref{fig:h4-convergence} shows second-order temporal convergence for each tested viscosity, with comparable errors over the full viscosity range, and fourth-order spatial convergence. These observations agree with the viscosity-uniform temporal estimate and the fixed-viscosity fourth-order spatial estimate in Theorem~\ref{thm:full-error}.
\begin{figure}[htbp]
\centering
\begin{subfigure}[t]{0.48\textwidth}
\centering
\includegraphics[width=\linewidth]{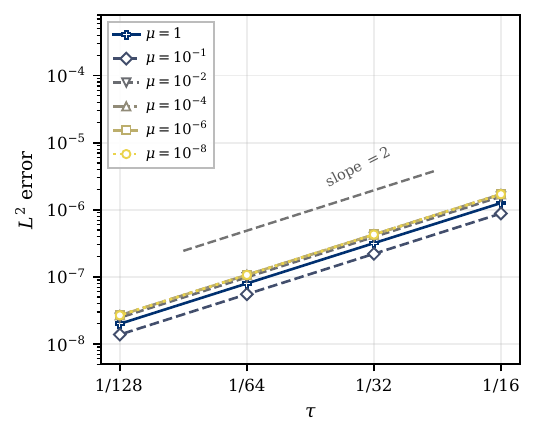}
\vspace{-20pt}
\caption{Temporal convergence}
\label{fig:h4-temporal-convergence}
\end{subfigure}
\begin{subfigure}[t]{0.48\textwidth}
\centering
\includegraphics[width=\linewidth]{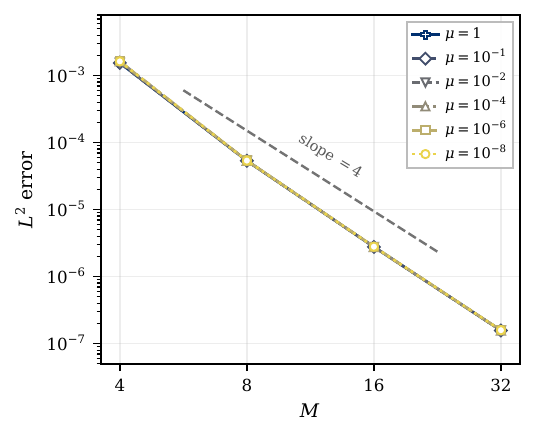}
\vspace{-20pt}
\caption{Spatial convergence}
\label{fig:h4-spatial-convergence}
\end{subfigure}
\vspace{-10pt}
\caption{Temporal and spatial convergence for Fourier-Galerkin method.}
\label{fig:h4-convergence}
\end{figure}

\FloatBarrier

 Figure~\ref{fig:energy-decay-mu} shows the discrete energy and the dissipation rate
\[\mathcal D^{n+\frac12}
=-\frac{E(U^{n+1})/E(U^0)-E(U^n)/E(U^0)}{\tau}
\]
for $M=32$, $T=0.5$, and $N=512$. The inset in panel~(a) resolves the relative energy change for $\mu=10^{-8},10^{-6},10^{-4}$. In all cases, the energy decreases monotonically and dissipation rate remains nonnegative, with stronger dissipation for larger $\mu$.
\begin{figure}[htbp]
\centering
\begin{subfigure}[t]{0.48\textwidth}
\centering
\includegraphics[width=\linewidth]{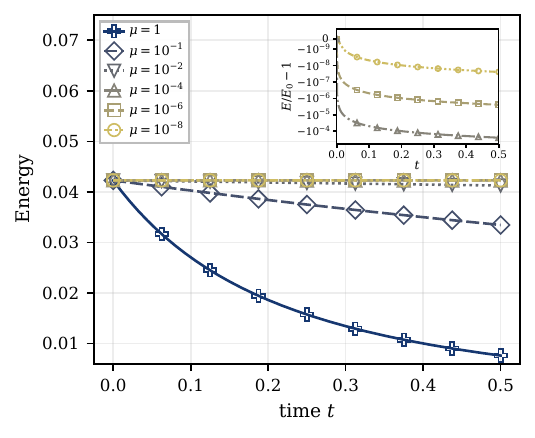}
\label{fig:energy-decay-normalized}
\end{subfigure}
\begin{subfigure}[t]{0.48\textwidth}
\centering
\includegraphics[width=\linewidth]{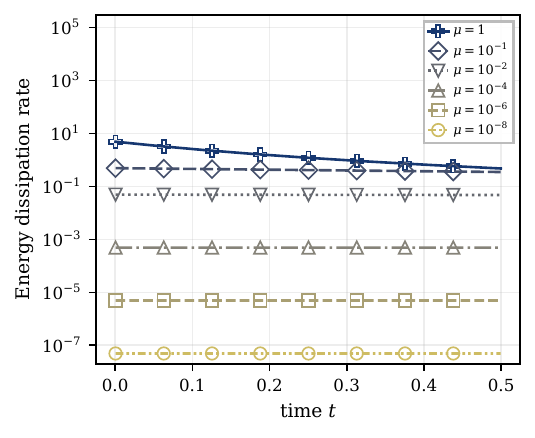}
\label{fig:energy-dissipation-rate}
\end{subfigure}
\vspace{-10pt}
\caption{Discrete energy dissipation over time.}
\label{fig:energy-decay-mu}
\end{figure}

\FloatBarrier
We denote the proposed method by LRI2 and compare it with LRI1 \cite{LMS}, ETD-RK2 \cite{CM}, CN-LE \cite{Ingram,IngramError}, and BDF2-LE \cite{AKR,GAN}, their structural properties are summarized in Table~\ref{tab:comparison-methods}. The regularity column lists
sufficient spatial conditions in the periodic setting.  All methods use the same Fourier space and initial data. 
For the temporal test, we evaluate the convergence in \eqref{eq:temporal-self-convergence} with $M=48$, $T=0.5$, and $N=16,32,64,128$. Figure~\ref{fig:accuracy-viscosity-h4} shows that LRI2 attains second-order convergence and gives the smallest $L^2$ error at every common step size for both $\mu=1$ and $10^{-6}$. Figure~\ref{fig:comparison-energy-h4} compares the energy decay properties of LRI2, LRI1, CN-LE, and BDF2-LE over three viscosity regimes. LRI1 exhibits additional numerical dissipation, whereas the kinetic energy of BDF2-LE need not decrease because its energy law applies to a modified energy.
\begin{table}[htbp]
\caption{Methods used in the comparison. }
\label{tab:comparison-methods}
\centering
\scriptsize
\setlength{\tabcolsep}{3pt}
\resizebox{\textwidth}{!}{%
\begin{tabular}{@{}lcc
p{0.15\textwidth}
p{0.24\textwidth}
p{0.16\textwidth}@{}}
\toprule [0.8pt]
Method & \shortstack{Time\\order} & \shortstack{Linear} & Diffusion
& \shortstack{Sufficient\\regularity} & Energy decay\\[2pt]
\midrule \\
\textbf{LRI2} & 2 & $\checkmark$ & Exact heat flow
& $u\in L^\infty(0,T;H^4)$ & \shortstack{$\checkmark$}  \\[8pt]
LRI1   & 1 & $\checkmark$  & Exact heat flow
& $u\in L^\infty(0,T;W^{2,d+\varepsilon})$ 
& $\checkmark$ \\[8pt]
ETD-RK2 & 2 &$\checkmark$  & Exact heat flow
& Not specified  & Not specified\\[8pt]
CN-LE   & 2 & $\checkmark$  & Crank--Nicolson
  & $u\in L^\infty(0,T;H^6)$ & $\checkmark$ \\[8pt]
BDF2-LE & 2 & $\checkmark$  & BDF2
  & $u\in L^\infty(0,T;H^6)$ 
  & Modified energy \\[4pt]
\bottomrule [0.8pt]
\end{tabular}%
}
\end{table}

\begin{figure}[htbp]
\centering
\begin{subfigure}[t]{0.48\textwidth}
\centering
\includegraphics[width=\linewidth]{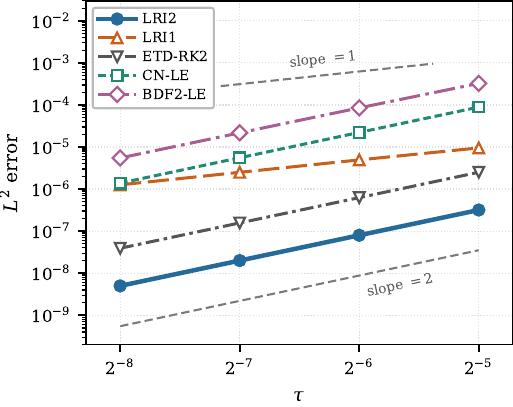}
\vspace{-12pt}
\caption{$\mu=1$}
\label{fig:accuracy-viscosity-h4-mu1}
\end{subfigure}\hfill
\begin{subfigure}[t]{0.48\textwidth}
\centering
\includegraphics[width=\linewidth]{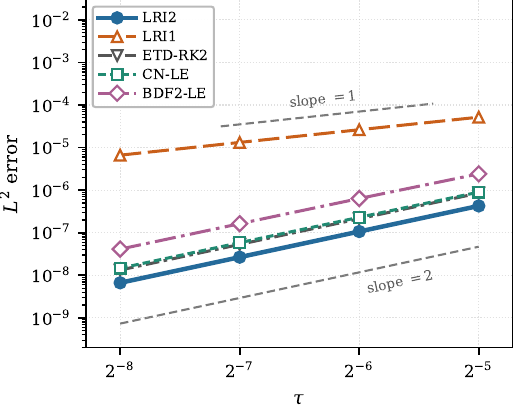}
\vspace{-12pt}
\caption{$\mu=10^{-6}$}
\label{fig:accuracy-viscosity-h4-mu1e-6}
\end{subfigure}
\vspace{-10pt}
\caption{Temporal convergence for methods used in the comparison}
\label{fig:accuracy-viscosity-h4}
\centering
\hspace{-12pt}
\begin{subfigure}[t]{0.35\textwidth}
\centering
\includegraphics[width=\linewidth]{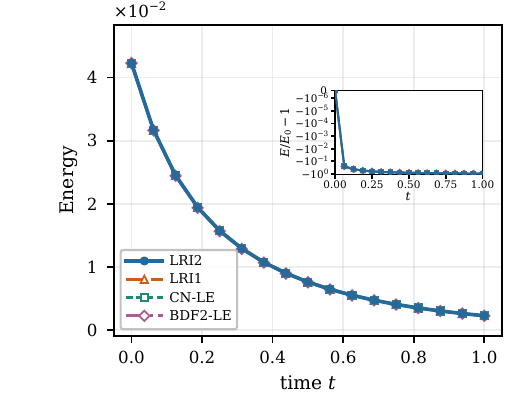}
\vspace{-14pt}
\caption{$\mu=1$}
\label{fig:comparison-energy-mu1}
\end{subfigure}
\hspace{-10pt}
\begin{subfigure}[t]{0.35\textwidth}
\centering
\includegraphics[width=\linewidth]{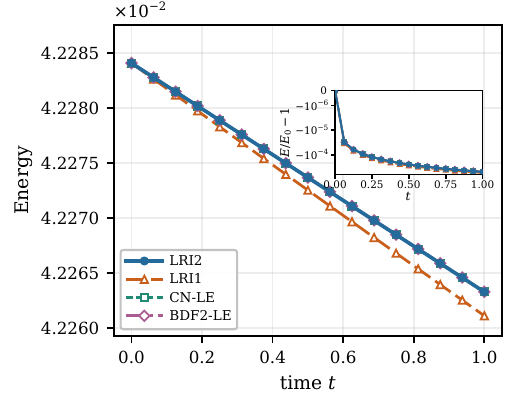}
\vspace{-14pt}
\caption{$\mu=10^{-4}$}
\label{fig:comparison-energy-mu1e-4}
\end{subfigure}
\hspace{-8pt}
\begin{subfigure}[t]{0.35\textwidth}
\centering
\includegraphics[width=\linewidth]{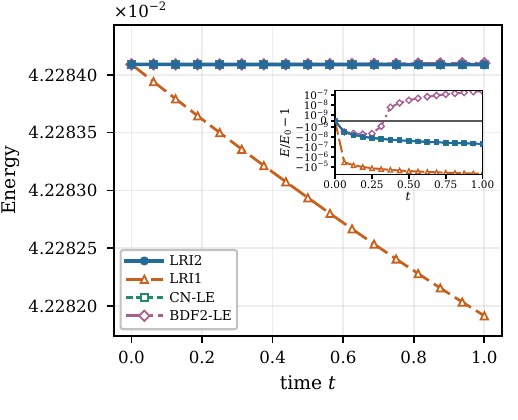}
\vspace{-14pt}
\caption{$\mu=10^{-8}$}
\label{fig:comparison-energy-mu1e-8}
\end{subfigure}
\hspace{-10pt}
\caption{Discrete energy for methods used in the comparison.}
\label{fig:comparison-energy-h4}
\end{figure}

\FloatBarrier

We also examine the finite element method from Section~\ref{sec:finite-element}
using the periodic $Q_2$--$Q_1$ Taylor--Hood finite element space. Although a 
error analysis for the finite element method
is not included, the numerical results in
Figure~\ref{fig:fe-convergence} exhibit second-order temporal convergence and
the optimal third-order spatial convergence in $L^2$. In addition,
Figure~\ref{fig:fe-energy-decay} illustrates the discrete energy-decay property
established in Section~\ref{sec:finite-element} over the tested range of
viscosities. These results support the applicability of the proposed time
integrator to conforming finite element discretizations.

\begin{figure}[htbp]
\centering
\begin{subfigure}[t]{0.48\textwidth}
\centering
\includegraphics[width=\linewidth]{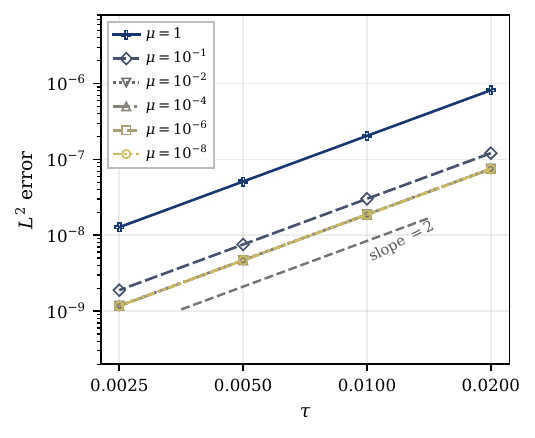}
\vspace{-20pt}
\caption{Temporal self-convergence}
\label{fig:fe-temporal-convergence}
\end{subfigure}
\begin{subfigure}[t]{0.48\textwidth}
\centering
\includegraphics[width=\linewidth]{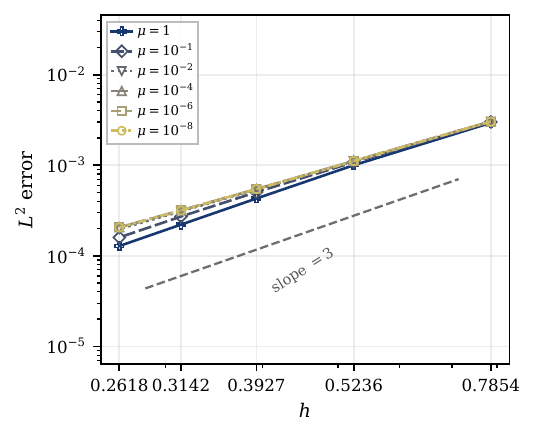}
\vspace{-20pt}
\caption{Spatial self-convergence}
\label{fig:fe-spatial-convergence}
\end{subfigure}
\vspace{-10pt}
\caption{Temporal and spatial convergence
for the finite element method. }
\label{fig:fe-convergence}
\centering
\begin{subfigure}[t]{0.48\textwidth}
\centering
\includegraphics[width=\linewidth]{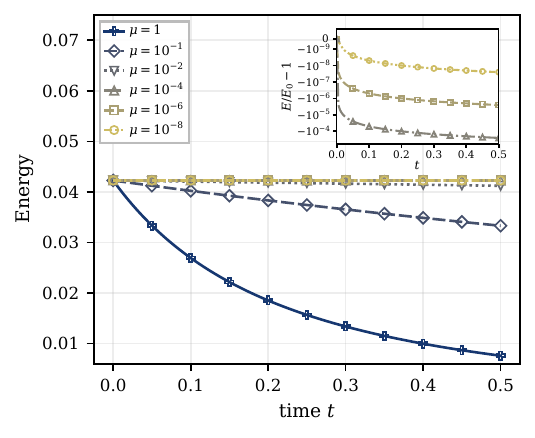}
\label{fig:fe-discrete-energy}
\end{subfigure}
\begin{subfigure}[t]{0.48\textwidth}
\centering
\includegraphics[width=\linewidth]{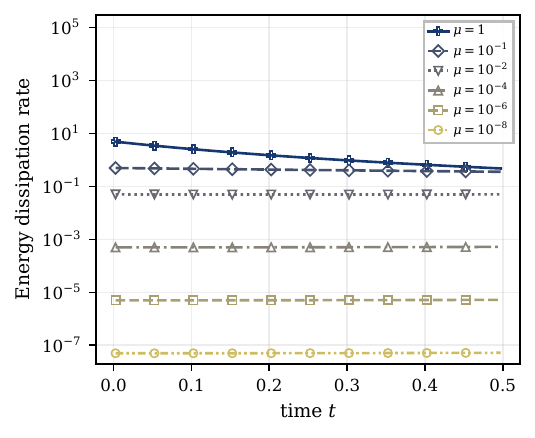}
\label{fig:fe-dissipation-rate}
\end{subfigure}
\vspace{-10pt}
\caption{Discrete energy dissipation over time for finite element method}
\label{fig:fe-energy-decay}
\end{figure}

\section{Conclusion}\label{sec:conclusions}
We have constructed a second-order exponential low-regularity integrator for
the incompressible Navier--Stokes equations by combining exact heat half-steps
with an energy-preserving predictor--corrector discretization of convection.
The method is uniquely solvable and satisfies an exact discrete kinetic-energy
law.  Assuming a uniform $H^4$ bound for the exact solution, the fully discrete
Fourier--Galerkin approximation converges in $L^2$ with error
 $O(\tau^2+M^{-3})$ for every $\mu>0$. The error constant
remains bounded as $\mu\downarrow0$, and no condition couples the time step to
the spectral cutoff.  For fixed $\mu>0$, the spatial convergence can be improved to the optimal order $O(M^{-4})$. The construction extends to conforming inf-sup
 stable finite element pairs through a skew-symmetric convection form. The
 numerical results confirm the second-order temporal accuracy and robustness
 with respect to the viscosity.

\bibliographystyle{siamplain}
\bibliography{references}
\end{document}